\documentclass[11pt]{article}
\usepackage{hyperref}
\usepackage[active]{srcltx}
\usepackage[curve]{xypic}
\usepackage{graphicx}

\usepackage[active]{srcltx}
\usepackage{epsfig,amsmath}
\usepackage[english]{babel}
\usepackage{amsmath,amsfonts,amssymb,amscd,color,epsfig,amsthm}

\newtheorem{newthm}{Theorem}
\newtheorem{theorem}{Theorem}[section]
\newtheorem{lemma}[theorem]{Lemma}
\newtheorem{proposition}[theorem]{Proposition}
\newtheorem{corollary}[theorem]{Corollary}

\newtheorem{definition}[theorem]{Definition}

\theoremstyle{remark}

\theoremstyle{plain}

\numberwithin{equation}{section}

\newcommand{\D}{\mathbb{D}}

\newcommand{\C}{\mathbb{C}}

\newcommand{\wt}{\widetilde}
\newcommand{\mb}{\mathbb}
\newcommand{\mc}{\mathcal}

\newcommand{\tu}{\textup}
\newcommand{\ol}{\overline}
\newcommand{\es}{\emptyset}
\newcommand{\mt}{\mapsto}

\newcommand{\mf}{\mathbf}

\newcommand{\olC}{\overline{\mathbb{C}}}

\def\CCC{{\cal C}}

\def\CCC{{\cal C}}

\def\C{\mbox{$\mathbb C$}}

\def\D{\mathbb D}

\def\lv{ \left(\begin{matrix} }
 \def\rv{\end{matrix}\right)}

\def\cal{\mathcal}

\def\dw{{\dw}}

\newcommand{\mylabel}[1]{\label{#1}}

\newcommand{\REFEQN}[1] { \begin{equation}\mylabel{#1} }
\newcommand{\ENDEQN}{\end{equation}}
\newcommand{\REFTHM}[1] { \begin{theorem}\mylabel{#1} }
\newcommand{\ENDTHM}{\end{theorem}}
\newcommand{\REFNTH}[1] { \begin{newthm}\mylabel{#1} }
\newcommand{\ENDNTH}{\end{newthm}}
\newcommand{\REFPROP}[1]{\begin{proposition}\mylabel{#1} }
\newcommand{\ENDPROP}{\end{proposition} }
\newcommand{\REFLEM}[1]{\begin{lemma}\mylabel{#1} }
\newcommand{\ENDLEM}{\end{lemma} }
\newcommand{\REFCOR}[1]{\begin{corollary}\mylabel{#1} }
\newcommand{\ENDCOR}{\end{corollary} }

\def\ov{\overline}

\title{{{\bf{A CHARACTERIZATION OF SIERPI\'{N}SKI CARPET RATIONAL MAPS}}}}
\author{
Yan Gao \footnote{partially supported by the grant no. 11501383 of NSFC.}
\and
Jinsong Zeng
\and
Suo Zhao\footnote{the corresponding author.}
}

\date{\today}

\begin{document}
\maketitle
\begin{abstract}
In this paper, we prove that a postcritically finite rational map with non-empty Fatou set is Thurston equivalent to an expanding Thurston map if and only if its Julia set is homeomorphic to the standard Sierpi\'{n}ski carpet.
\end{abstract}

\section{Introduction}

In Complex Dynamics, a central theme is to understand the global dynamics of the \emph{postcritically finite rational maps} (see Section \ref{expanding_thurston} for its definition).  In the case
of postcritically finite polynomials, Douady and Hubbard have introduced the so-called
Hubbard trees which capture their dynamical features \cite{DH2}. But for general rational maps, as far as we know, the overall understanding  has remained sketchy and unsatisfying (see e.g. \cite{CFKP, CG, CPT2}).

   When ignoring the complex structure, we consider a postcritically finite rational map as a postcritically finite \emph{branched covering} of the sphere $S^2$. Such maps are called {\it Thurston maps}.   Recently, M. Bonk-D. Meyer \cite{BM}, D. Meyer \cite{M1,M2}, 
   Z.Q. Li \cite{L} et al studied the dynamics of a kind of Thurston maps, called {\it expanding Thurston maps} (see Definition.~\ref{def:f}), from the aspects of combinatorics, topology, geometry and ergodic theory.

Thus, if we can establish a relation between  expanding Thurston maps and some class of rational maps (in the dynamical sense), one may, at least in principle, apply the methods and results used for expanding Thurston maps to the study of the corresponding rational maps.
In the level of topological conjugacy, it was shown in \cite[Proposition 2.3]{BM} that a rational map topologically conjugates to an expanding Thurston map if and only if its Julia set is the Riemann sphere.
In a weaker sense, Thurston introduced an equivalence relation among all Thurson maps, called {\it Thurston equivalence} (see Definition.~\ref{thurston-equivalence}).
Then a natural question is:

\textit{ What kind of postcritically finite rational maps with non-empty Fatou sets are Thurston equivalent to expanding Thurston maps?}

Our answer to this question is as follows.
\begin{theorem}\label{main_thm}
 A postcritically finite rational map with non-empty Fatou set is Thurston equivalent to an expanding Thurston map if and only if
its Julia set is homeomorphic to the standard Sierpi\'{n}ski carpet.
\end{theorem}

To verify this theorem, we will recall some basic definitions and results in Section \ref{def_expanding}, and prove a series of lemmas about \emph{homotopy} and \emph{isotopy} in Section \ref{homotopy_equivalent} and Section \ref{sec-appendix}. The detailed proof of Theorem \ref{main_thm} is left in Section \ref{proof}.

 We will end the introduction with two remarks.

{\noindent \it 1.} There are many examples of postcritically finite rational maps with Julia sets homeomorphic to the standard Sierpi\'{n}ski carpet (see e.g. \cite[Appendix]{milnor:tanlei} and \cite{XQY}), and these Julia sets has the quasisymmetric rigidity \cite{BLM}. Conjecturally, the components of these rational maps are relatively compact in the space of rational functions up to M\"{o}bius conjugation. \cite[Question 5.3]{Mc1}

{\noindent \it 2.} In the proof of the main theorem, we use the following trick:  we first construct a {homotopy} $H:S^2\times I\to S^2~{\rm rel.}~P$ between two homeomorphisms, and then modify it to an {isotopy} relative to $P$, where $P$ is a finite subset of $S^2$. We emphasize that this result is generally false; see Section \ref{homotopy_equivalent} for a counterexample and detailed discussion.

{\noindent \bf Acknowledgement.} We really appreciate D. Meyer and P. Ha\"{\i}ssinsky very much for their helpful discussions and valuable suggestions. We also thank D. Margalit and A. Hatcher for enthusiastically answering our questions about surface topology.

\section{Preliminary}\label{def_expanding}
Here we present some notations and elementary background that will be used in
this paper. More details can be found in \cite{BM, CPT, Mi1}.
\subsection{Notations}
The $2$-sphere is denoted by $S^2$, the Riemann sphere  by
$\olC$ and the open unit disk by $\mathbb{D}$. The closure and interior of a subset $K\subset S^2$ is denoted by ${\rm cl}(K)$ and ${\rm int}(K)$ respectively.
The spherical metric on $S^2$ is $\sigma=\frac{2|dz|}{1+|z|^2}$.
The set of critical points of a branched covering
$F$ is denoted by ${\rm crit}(F)$ and the set of postcritical points by
${\rm post}(F)$. The Julia
set of a rational map $f$ will be denoted by $\mathcal{J}_f$; the Fatou
set is $\mathcal{F}_f$.

\subsection{Expanding Thurston maps}\label{expanding_thurston}
Let $F:X\to Y$ be a continuous map between two domains $X,Y\subset S^2$. The map $F$ is called a \emph{branched covering} if for each point $p\in X$, there exist an integer $n\geq1$, open neighborhoods $U$ of $p$ and $V$ of $q:=F(p)$, and orientation-preserving homeomorphisms $\phi:U\to \mb{D}$ and $\psi:V\to\mb{D}$ with $\phi(p)=\psi(q)=0$ such that
$\psi\circ F\circ\phi^{-1}(z)=z^n$ for all $z\in\mb{D}$.

The integer $\tu{deg}_F(p):=n\geq 1$ is called the \emph{local degree} of $F$ at $p$. A point $c$ with $\tu{deg}_F(c)\geq 2$ is called a \emph{critical point} of $F$.
A branched covering without critical points is called a \emph{covering}.

Let $F:S^2\to S^2$ be a branched covering. The set of critical points of $F$ is denoted by ${\rm crit}(F)$, and
the \emph{postcritical set} ${\rm post}(F)$ is defined as
$$\textup{post}(F):=\ol{\cup_{n\geq 1}F^{n}({\rm crit}(F))}.$$
The map $F$ is called \emph{postcritically finite} if $\#\textup{post}(F)<\infty$.

\begin{definition}\label{def:f}
  A \emph{Thurston map} is an orientation-preserving, postcritically
  finite, branched covering of $S^2$.
We fix a base metric $\rho$ on $S^2$
that induces the standard
topology on $S^2$. Consider a Jordan curve $\mathcal{C}\supset {\rm post}(f)$. The Thurston map
  $F$ is
  called \emph{expanding} if
    \begin{equation*}   \label{def:fexpanding}
      {\rm mesh}~F^{-n}(\CCC) \to 0 \text{ as } n\to \infty.
    \end{equation*}
where  ${\rm mesh}~F^{-n}(\CCC)$ denotes the maximal diameter of a
  component of $S^2\setminus F^{-n}(\CCC)$.
\end{definition}

It was shown in \cite{BM} that the expansionary property is  independent of the choice of the Jordan curve $\CCC$ (\cite[Lemma~6.1]{BM}), and  the base metric $\rho$ on $S^2$ as long as it induces the standard topology on $S^2$ (\cite[Proposition~6.3]{BM}).

\subsection{Partitions of $S^2$ induced by Thurston maps}\label{cell-decomposition}

Let $F:S^2\to S^2$ be a Thurston map
and fix a Jordan curve $\CCC\subset S^2$ with ${\rm post}(F) \subset \CCC$.
The closure of one of the two components of $S^2\setminus \CCC$
is called a \emph{$0$-tile} (relative to $(F,\CCC)$). Similarly, we call the closure
of one component of $S^2\setminus F^{-n}(\CCC)$ an
$n$\emph{-tile} (for any $n\geq 0$). The set of all $n$-tiles
is denoted by
$\mathbf{X}^n(\CCC)$. For any $n$-tile $X$, the set $F^n(X)=X^0$ is a $0$-tile
and
\begin{equation}
  \label{eq1}
  F^n\colon X\to X^0\quad \text{is a homeomorphism,}
\end{equation}
see \cite[Proposition 5.17]{BM}.  This means
in particular that each $n$-tile is a closed Jordan domain. The
definition of ``expansion'' implies that $n$-tiles become
arbitrarily small. Clearly, for each $n\geq0$, all $n$-tiles relative to $(F,\CCC)$ form a partition of $S^2$.


\subsection{Postcritically finite  Sierpi\'{n}ski carpet rational map}\label{rational map}

Let $f$ be a postcritically finite rational map. It was known that $\mathcal{J}_f$ is connected and locally connected (\cite{Mc2, Mi1}). The connectedness implies that each Fatou component is simply connected; and the local connectedness implies  the following lemma (see \cite[Theorem VI.4.4]{Why2}).
\begin{lemma}
  \label{small}
  For any $\epsilon>0$, there are finitely many Fatou components $U$ with
  ${\rm diam}_\sigma(U)\geq \epsilon$.
\end{lemma}

Let $U$ be a Fatou component of $f$. From \emph{Boettcher's theorem} it follows that there is a conformal map $\eta_{U}:\D\to U$ and some power $d_{U}$
such that $f\circ \eta_{U}(z)=
\eta_{f(U)}(z^{d_{U}})$ for all $z\in\D$.
Since the Julia set $\mathcal{J}_f$ is locally
connected, it follows from Carath\'eodory's theorem that the
conformal map $\eta$ extends to a continuous and surjective map
$\eta_{U}:\overline{\D}\to\overline{U}$.
An \emph{internal ray} of $U$ is the image $\eta_{U}([0,1)\theta)$ for unit number $\theta\in\partial\D$.
Note that internal rays are mapped to internal rays under $f$.

A set $S\subset \overline{\C}$ is called a \emph{(Sierpi\'{n}ski)} \emph{carpet} if it is homeomorphic to the standard Sierpi\'nski carpet. By Whyburn's characterization \cite{Why}, a set $S\subset\olC$  is a carpet if and only if it can be written as $S=\olC\setminus\cup_{n\geq 1}D_n$, where all $D_n$ are Jordan domains with pairwise disjoint closures, such that the interior of $S$ is empty and the spherical diameters $\tu{diam}_\sigma(D_n)\to 0\tu{ as }n\to\infty$.

We say a postcritically finite rational map to be a postcritically finite \emph{carpet rational map} if its Julia set is a carpet. This means that each Fatou component  is a Jordan
domain and distinct components of the Fatou set have  disjoint
closures. Furthermore, the boundary of a component of the  Fatou
set cannot contain  postcritical points.

\section{Homotopy, isotopy and Thurston equivalent}\label{homotopy_equivalent}
\begin{definition}[Relative homotopy and isotopy]
Let $X, Y$ be topological spaces and $A$ be a subset of $X$ (maybe empty). Let $\phi,\psi$ be continuous maps from $X$ to $Y$. We say that $\phi$ and $\psi$ are \emph{homotopic rel.}~$A$ if there exists a continuous map
$H: X\times[0,1]\to Y$, called a \emph{homotopy rel.}~$A$, such that
$$H(x,0)=\phi(x),~H(x,1)=\psi(x)~~\forall x\in X, \text{ and } H(x,t)=\phi(x)~~ \forall x\in A,~\forall t\in[0,1].$$
If the map $H|_{X\times t}:X\to Y$ is a homeomorphism for each $t\in[0,1]$, we call $H$ an \emph{isotopy rel.}~$A$.
\end{definition}
Let $H:X\times [0,1]\to Y$ be a homotopy. For simplity, we usually denote the map $H|_{X\times t}:X\to Y$ by $H_t:X\to Y$, where $t\in[0,1]$.

\begin{lemma}\label{homotopy}
$(1)$ Let $I=[0,1]$ and $\psi:I\to I$ be a continuous surjective map with $\psi(0)=0$ and $\psi(1)=1$. Then $\psi$ is homotopic to $id_{I}$  rel. $\{0,1\}$.

$(2)$ Let $A$ be a subset of the unit circle and  $\phi:\ov{\D}\to \ov{\D}$ be a continuous and surjective map. If the restriction $\phi|_{\partial\D}$ is homotopic to $id_{\partial\ov{\D}}$ relative to $A$, then $\phi$ and $id_{\ov{\D}}$ are homotopic rel. $A$.
\end{lemma}

\begin{proof}
$(1)$ Define $H:I\times I\to I$ by
\[H(s,t)=ts+(1-t)\psi(s)\]
for $s,t\in I$. Then $H_t(0)=0$ and $H_t(1)=1$ for $t\in I$, and $H_0=\psi$, $H_1=id_I$. Hence $\psi$ and $id_I$ are homotopic rel. $\{0,1\}$ by $H$.

$(2)$ Let $h:\partial \mb{D}\times I\to \partial \mb{D}$ be the homotopy between $id|_{\partial\mb{D}}$ and $\phi|_{\partial\mb{D}}$ relative to $A$ with $h(\cdot,0)=id|_{\partial\mb{D}}$, $h(\cdot,1)=\phi|_{\partial\mb{D}}$ and $h(x,t)=x$ for all $x\in A, t\in I$. We obtain the desired homotopy by a small change of the ``Alexander trick''.

For $0<t<1$, we define
\[H(z,t)=\left\{
           \begin{array}{ll}
             t\phi(z/t), & \hbox{$0\leq |z|<t$;} \\[5pt]
             t\phi(z/|z|)\big(1-\dfrac{|z|-t}{1-t}\big)+h(z/|z|,t)\dfrac{|z|-t}{1-t}, & \hbox{$t\leq |z|\leq 1$.}
           \end{array}
         \right.
\]

We complementarily define $H(z,0):=z$ for $t=0$ and $H(z,0):=\phi(z)$ for $t=1$, then
 $H:\ol{\mb{D}}\times I\to \ol{\mb{D}}$ is a homotopy between $id$ and $\phi$ relative to $A$.
\end{proof}

A classical result about the modification of homotopy to isotopy in surfaces is due to D. B. A. Epstein \cite{Ep}; see also \cite[Theorem. 1.12]{FM}.

\begin{theorem}[Epstein]\label{homo-iso}
Let $S$ be a surface obtained from an orientable closed surface by removing $b\geq0$ open disks and $n\geq0$ points with disjoint closures, and $h$ an orientation preserving homeomorphism of $S$. If $h$ and $id_S$ are homotopic rel. $\partial S$, then they are isotopic rel. $\partial S$.
\end{theorem}

One may ask when marking a finite set $P$ in the surface $S$ (defined as in the theorem above), are two orientation preserving  homeomorphisms of $S$ that are homotopic rel.~$\partial S\cup P$ still isotopic rel.~$\partial S\cup P$? The answer  is \textsc{NO} in general. Here is a simple counterexample:

Choose $S=\overline{\mathbb{D}}$, the closed unit disk, and let $P\subset \mathbb{D}$ contain at least two points. Let $h$ be a Dehn twist on   $\overline{\mathbb{D}}$ along a Jordan curve surrounding $P$. It is known that $h$ is not isotopic to $id_{\overline{\mathbb{D}}}$ rel.~$\partial\mathbb{D}\cup P$. But
$H(z,t)=tz+(1-t)h(z),t\in[0,1],z\in \overline{\mathbb{D}}$ is a homotopy rel.~$\partial\mathbb{D}\cup P$ between $h$ and $id_{\overline{\mathbb{D}}}$. A similar counterexample can be given on $S^2$ with at least four marked points.

So, in general, two homeomorphisms of an orientable surface  homotopic relative to marked points are not necessarily isotopic relative to the marked points. However, if the homotopy is well chosen, the conclusion holds. We leave the proof of the following lemma in the appendix.

\begin{lemma}\label{good-homo}
Let $P$ be a finite set in an orientable surface $S$, and $H:S\times I\to S$~rel.~$P\cup\partial S$ be a homotopy  such that $H_0=id_S$ and $h:=H_1$ is an orientation preserving homeomorphism. If each $K_p:=\cup_{t\in[0,1]}H^{-1}_t(p)$ is \emph{contractible within $S\setminus P$}, i.e., all but one components of $(S\setminus P)\setminus K_p$ are simply-connected, and they are pairwise disjoint,  then $\phi$ is isotopic to $id_{S}$ rel.~$P\cup \partial S$.
\end{lemma}

At the end of this section, we introduce the concept of Thurston equivalent.
\begin{definition}[Thurston equivalent]\label{thurston-equivalence}
Two Thurston maps $F, G$ on $S^2$ are said to be \emph{Thurston equivalent} if there exist homeomorphisms $\psi,\phi:S^2\to S^2$ that are isotopic rel.~${\rm post}(F)$ and satisfy $G\circ\psi=\phi\circ F$, that is, the following commutative diagram follows:
\[
\begin{CD}
{S^2} @>{F} >> {S^2}  \\
@V{\psi}VV @ V{\phi} VV   \\
{S^2} @>{G} >> {S^2} \\
\end{CD}
\]
\end{definition}

\section{A characterization of  carpet rational maps}\label{proof}
The objective of this section is to prove Theorem \ref{main_thm}. We will first summarize the idea  ( Section\ \ref{outline}) and  then give the detailed proof ( Sections\ \ref{expanding-quotient} and \ref{main-proof}).

\subsection{The outline of the proof}\label{outline}
Let $f$ be a postcritically finite rational map, and $F$ a Thurston map.

For the necessity, by repeatedly using the isotopy lifting theorem, we obtain a sequence of homeomorphisms $\{\phi_n\}$ such that $\phi_{n}\circ f=F\circ \phi_{n+1}$ and $\phi_n$ is isotopic to $\phi_{n+1}$ rel. $f^{-n}({\rm post}(f))$ for all $n\geq 0$. This sequence of homeomorphisms converges to a semi-conjugacy $h$ from $f:\widehat{\mathbb{C}}\to\widehat{\mathbb{C}}$ to $F:S^2\to S^2$ by the expansionary property of $F$. With the properties of this semi-conjugacy $h$, we can prove that the Julia set of $f$ is a Sierpi\'{n}ski carpet.

The sufficiency  proceeds as follows. Suppose that $f$ has  Sierpi\'{n}ski carpet Julia set. By collapsing the closure of each Fatou component to a point, we obtain the quotient map $\pi:\ol{\mathbb{C}}\to S^2$, by which the rational map $f$ descends to an expanding Thurson map (This assentation is proven in   Section\ \ref{expanding-quotient}). This yields a semi-conjugacy $\pi$ from the rational map $f$ to an expanding Thurston map $F$. We can carefully choose a homeomorphism $\psi$ in the homotopy class of $\pi$ rel.~${\rm post}(f)$ such that $\psi$ has a lift $\phi$ along $f$ and $F$, i.e., $F\circ\psi=\phi\circ f$ on $\widehat{\mathbb{C}}$, and the homeomorphism $\phi$ is homotopic to $\pi$ rel.~${\rm post}(f)$. We then get a homotopy rel. ${\rm post}(f)$ between $\psi$ and $\phi$ by concatenating the homotopy between $\psi,\pi$ and that between $\pi,\phi$. This homotopy turn out to satisfy the properties of Lemma \ref{good-homo}. It follows that $\phi$ and $\psi$ are isotopic rel. ${\rm post}(f)$.

\subsection{The expanding quotient}\label{expanding-quotient}

 We will show in this part that any postcritically finite carpet rational map  can be semi-conjugated to an expanding Thurston map.
The base of this fact is  Moore's Theorem.

\begin{lemma}[Moore \cite{Moo}] \label {moore}
Let $\equiv$ be an equivalence relation on  $\mathrm{2}$-sphere $S^2$ satisfying

$(1)$ the equivalence relation $\equiv$ is closed,

$(2)$ each equivalence class of $\equiv$ is a compact connected subset of $S^2$,

$(3)$ the complement of each equivalence class of $\equiv$ is a connected subset of $S^2$,

$(4)$ there are at least two distinct equivalence classes.

\noindent Then the quotient space $S^2/_\equiv$ is homeomorphic to $S^2$.
\end{lemma}

Let $f$ be a postcritically finite rational map with Sierpi\'{n}ski carpet Julia set. We then define an equivalence relation $\sim$ on $\olC$: for any $z,w\in \widehat{\mathbb{C}}$,
$z\sim w$  if and only if either $z=w$ or $z, w$  belong to the closure of a common Fatou component.

We claim that the equivalence relation $\sim$ satisfies the $4$ properties of Lemma \ref{moore}.
Clearly, the properties $(2), (3), (4)$ holds. To check the property (1), it suffices to show that given
  two convergent sequences $(z_n)_{n\geq1}$ and $(w_n)_{n\geq1}$ in $\widehat{\mathbb{C}}$ with $z_n\sim w_n$ for all $n\geq1$ it
  follows that $\lim z_n \sim \lim w_n$. This is clear in the
  case when for sufficiently large $n$ the points $z_n$ and $w_n$ are
  contained in some fixed equivalence class, since each
  equivalence
  class is compact. Otherwise, we may assume that for distinct
  $n,m\geq1$ the points $z_n$ and $z_m$ are contained in
  distinct equivalence classes. In this case, the diameter of the equivalence class containing $z_n$ becomes arbitrarily small as $n\to \infty$ by Lemma \ref{small}, then we have $\lim z_n = \lim w_n$. Thus $\sim$ is closed.

Using Lemma \ref{moore}, the quotient space
$$ \overline{\C}/_\sim= \{[z] : \text{$[z]$ is the equivalence class of $z$, }z \in\overline{\C}\}$$
with the quotient topology is homeomorphic to $S^2$. We identify $\mathbb{C}/_\sim$ with $S^2$ so that the quotient map can be written as the continuous map  $\pi: \olC\to S^2$. This map is also uniformly continuous with respect to the spherical metric $\sigma$. 

Since the equivalence relation $\sim$ is \emph{$f$-invariant}, i.e., $z\sim w\Rightarrow f(z)\sim f(w)$, then the map $f$ descends to a map $F$ defined by $F(x)=\pi\circ f\circ\pi^{-1}(x)$ for all $x\in S^2$, that is, the following commutative diagram holds:
\begin{equation}\label{commute-1}
\begin{CD}
{\widehat{\mathbb{C}}} @>{f} >> {\widehat{\mathbb{C}}}  \\
@V{\pi}VV @ VV{\pi}V   \\
{S^2} @>{F} >> {S^2} \\
\end{CD}
\end{equation}
Furthermore, note that $\sim$ is also \emph{strongly invariant}, i.e., the image of any equivalence class is an equivalence class, we then have the following result, see \cite[Corollary~13.8]{BM} for a proof.

\begin{lemma}[Properties of $F$]\label{properites_F}
The map $F$ is a Thurston map, fullfilling that $\textup{crit}(F)=\pi(\textup{crit}(f))$, $\tu{post}(F)=\pi(\tu{post}(f))$ and $\textup{deg}_F(\pi(c))=\textup{deg}_f(c)$ for any $c\in \tu{crit}(f)$.
\end{lemma}

We now remain to show the expansion of $F$.
The following topological result is needed.


Let $f$ be a postcritically finite rational map with Sierpi\'{n}ski carpet Julia set.  An arc or a Jordan curve in $\widehat{\mathbb{C}}$ is called \emph{regulated} (with respect to $f$) if its intersection with the closure of each Fatou component is either empty or a connected set, i.e., one point or one arc.

\begin{lemma}\label{regulated_gamma}
Let $f$ be a postcritically finite rational map with Sierpi\'{n}ski carpet Julia set.
\begin{enumerate}
\item Let $\wt{P}:=\{\wt{p}_1,\cdots,\wt{p}_N\}$ be a finite set in $\olC$ with $\wt{P}\cap\partial U=\emptyset$ and the cardinality $\#(\wt{P}\cap U)\leq 1$ for each Fatou component $U$. Then there exists a regulated Jordan curve $\wt{\CCC}\subset\widehat{\mathbb{C}}$ containing the set $\wt{P}$.
\item Let $\wt{\CCC}$ be a regulated Jordan curve  and $\wt{V}_0,~\wt{V}_1$  the two components of $\ol{\mathbb{C}}\setminus\wt{\CCC}$. Then $\CCC:=\pi(\wt{\CCC})$ is a Jordan curve, and $\pi({\rm cl}(\wt{V}_0)),~\pi({\rm cl}(\wt{V}_1))$ are the closures of the two components of $S^2\setminus \CCC$.
\end{enumerate}
\end{lemma}
\begin{proof}
{\it1. }Remember that $ \pi:\ol{\mathbb{C}}\to S^2$ is the quotient map defined above. We set $E=\{\pi(U)\mid U\in\mathcal{F}_f\}$ and $P=\{p_k:=\pi(\wt{p}_k)\mid 1\leq k\leq N\}$.  We first claim that there exist closed disk neighborhoods $D_k$ for each point $p_k\in P$ such that they are pairwise disjoint and their boundaries avoid $E$. To see this, notice that $\big\{S_{r,k}:=\{x\in S^2\mid\sigma(x,p_k)=r\}\big\}_{r>0}$ is a uncountable family of pairwise disjoint sets and $E(=\pi(\mathcal{F}_f))$ is countable. So we may choose sufficiently small $r_k$ for each $k\in\{1,\ldots,N\}$ such that $S_{r_k,k}\cap E=\emptyset$ and $S_{r_i,i}\cap S_{r_j,j}=\emptyset$ ($i\not=j$). The neighborhoods $D_k$ are defined as $D_k:=\{x\in S^2\mid \sigma(x,p_k)\leq r_k\}$.

 We then claim that there are  pairwise disjoint open arcs $\gamma_1,\ldots,\gamma_N\subset S^2\setminus(\cup_{1\leq k\leq N}D_k)$  such that $\gamma_k$ joins $D_k$,$D_{k+1}$ with $D_{N+1}:=D_1$ and  $\gamma_k\cap E=\emptyset$ for each $k\in\{1,\ldots,N\}$. Indeed, it is easy to find a sequence of pairwise disjoint open arcs $e_1,\ldots, e_N\subset S^2\setminus(\cup_{1\leq k\leq N}D_k)$ such that $\gamma_k$ joins $D_k$ and $D_{k+1}$. Moreover, for each $k$ we may also choose a small neighborhood $A_k$ of $e_k$ within $S^2\setminus(\cup_{1\leq k\leq N}D_k)$ so that $A_1,\ldots,A_N$ are still pairwise disjoint. Let $h_k:A_k\to\C$ be an injective map with $h_k(e_k)=(0,1)$, the open unit interval. Since $h_k(A_k)$ contains an uncountable family of pairwise disjoint horizontal intervals and $h_k(A_k\cap E)$ is countable, we may choose a horizontal interval in $h_k(A_k)$ disjoint with $h_k(E)$ such that its preimage by $h_k$ joins $D_k$ and $D_{k+1}$. We denote this arc by $\gamma_k$. Then the arcs $\gamma_1,\ldots,\gamma_N$ satisfy the requirements in the claim.
For each $k\in\{1,\ldots,N\}$, set $a_k$ and $b_k$ the intersection of $\partial D_k$ with $\gamma_{k-1}$ and $\gamma_k$ respectively.

By the two claims above, the lifts $\wt{\gamma}_k:=\pi^{-1}(\gamma_k)$ (\emph{resp}. $\wt{S}_k:=\pi^{-1}(\partial D_k)$), $k\in\{1,\ldots,N\}$, are pairwise disjoint open arcs (\emph{resp}. Jordan curves) in $\ol{\mathbb{C}}\setminus \ol{\mathcal{F}_f}$. They are therefore regulated. Besides, we also see that the boundary of $\wt{D}_k:=\pi^{-1}(D_k)$ is exactly $\wt{S}_k$, and the intersection of $\partial \wt{D}_k$ and $\cup_{k=1}^N{\rm cl}(\wt{\gamma}_{k})$ are two points $\wt{a}_k:=\pi^{-1}(a_k)$ and $\wt{b}_k:=\pi^{-1}(b_k)$. Therefore, to obtain a regulated Jordan curve containing $\wt{P}$, it is enough to select a regulated arc $\wt{\alpha}_k$ in each $\wt{D}_k$ that passes through the point $\wt{p}_k$ and joins the points $\wt{a}_k,\wt{b}_k\in \partial \wt{D}_k$. This can be easily done if one note that each $\wt{J}_k:=\wt{D}_k\cap \mathcal{J}_f$ is a Sierpi\'{n}ski carpet and it is mapped onto the standard carpet  by a self-homeomorphism of $S^2$.
Finally, the set $\wt{\CCC}:=(\cup_{k=1}^N\wt{\gamma}_k)\cup(\cup_{k=1}^N\wt{\alpha}_k)$ is a regulated Jordan curve  containing the set $\wt{P}$.

{\noindent \it 2}.  By the definition of the regulated Jordan curves, the map $\pi$ collapses one point or one arc on $\wt{\CCC}$ to a point. It follows from
 \cite[Lemma 13.30]{BM} that $\CCC$ is a Jordan curve.

 Let $x\neq y$ belong to a component of $S^2\setminus \CCC$. Then  $\pi^{-1}(x)$ and $\pi^{-1}(y)$ are contained in a common component of $\widehat{\mathbb{C}}\setminus \wt{\CCC}$. Otherwise, we pick an arc $\gamma$ in $S^2\setminus\CCC$ joining $x$ and $y$. By \cite[Lemma 3.1]{CPT}, the set $\pi^{-1}(\gamma)$ is a continuum containing $\pi^{-1}(x)$ and $\pi^{-1}(y)$.  It hence intersects $\wt{\CCC}$. Consequently, we get  $\gamma\cap\CCC\neq\es$, a contradiction.

By this fact, we can label the two components of $S^2\setminus \CCC$ by $V_0$ and $V_1$ such that $\pi^{-1}(V_0)\subset \wt{V}_0$ and $\pi^{-1}(V_1)\subset \wt{V}_1$. It implies that $\pi({\rm cl}(\wt{V}_i))\subset {\rm cl}({V}_i)$ for $i=0,1$. Since $\pi$ is surjective, we have $\pi({\rm cl}(\wt{V}_0))={\rm cl}(V_0)$ and $\pi({\rm cl}(\wt{V}_1))={\rm cl}(V_1)$
\end{proof}

Applying 1 of Lemma \ref{regulated_gamma} to the case of $\wt{P}:=\tu{post}(f)$, we obtain a regulated Jordan curve passing through
$\textup{post}(f)$. Fix this curve and denote it by $\CCC_f$. By 2 of Lemma \ref{regulated_gamma}, the set $\CCC_F:=\pi(\CCC_f)$ is a Jordan curve in $S^2$ containing $\tu{post}(F)$. We denote by $\mf{X}^n(\CCC_f)$ and $\mf{X}^n(\CCC_F)$ the set of $n$-tiles relative to $(f,\CCC_f)$ and $(F,\CCC_F)$ respectively.  The following lemma implies a correspondence between them.
\begin{lemma}\label{corresponding}
For any $n\geq 0$, the map $\Psi_n:\mf{X}^n(\CCC_f)\to \mf{X}^n(\CCC_F)$ defined by $\Psi(\wt{X}^n)\mt \pi(\wt{X}^n)$
is well-defined and one to one.
\end{lemma}
\begin{proof}
Note that $f^n$ and $F^n$ are branched covering of degree $d^n$, then the cardinality of  $\mf{X}^n(\CCC_f)$ and $\mf{X}^n(\CCC_F)$ are the same, equal to $2d^n$.
For any $\wt{X}^n\in\mf{X}^n(\CCC_f)$, by (\ref{eq1}) the restriction $f^n:\wt{X}^n\to \wt{X}^0\in\mf{X}^0(\CCC_f)$ is a homeomorphism. It follows that $\partial \wt{X}^n_k$ is a regulated Jordan curve as well.
 The equation $F^n\circ\pi=\pi\circ f^n$ implies that $\pi(f^{-n}(\mathcal{C}_f))=F^{-n}(\mathcal{C}_F)$. Then an argument similar to the one used in the proof of 2 of Lemma \ref{regulated_gamma} shows  that, for each $X^n_k\in \mf{X}^n(\CCC_F)$, the set $\pi^{-1}(\tu{int}(X^n_k))$ is contained in a unique $n$-tile $\wt{X}^n_{\tau(k)}$ in $\mf{X}^n(\mathcal{C}_f)$ with $\tau(k)\neq \tau(k')$ if $k\neq k'$, and $\pi(\wt{X}^n_{\tau(k)})=X^n_k$. This fact gives  an one to one correspondence between $\mf{X}^n(\mathcal{C}_f)$ and $\mf{X}^n(\mathcal{C}_F)$ by $\wt{X}^n_{\tau(k)}\mapsto X^n_k$. The lemma is proved.
\end{proof}

Let $\mc{K}_0$ be the union of all \emph{postcritical Fatou components}, i.e., the Fatou components containing the postcritical points of $f$, and set $\mc{K}_n:=f^{-n}({\mc{K}}_0)$  for all $n\geq 0$. 
\begin{lemma}\label{max_diameter}
The maximum of $\tu{diam}_\sigma(\wt{X}^n\setminus\mc{K}_n)$ with $\wt{X}^n\in \mf{X}^n(\CCC_f)$ converges to $0$ as $n\to\infty$.
\end{lemma}
\begin{proof}
For any $n\geq0$ and $\wt{X}^n_k\in\mf{X}^n(\CCC_f)$, we set $U^n_k:=\tu{int}(\wt{X}^n_k\setminus\mc{K}_n)$. As $\partial \wt{X}^n_k$ is regulated, it follows that $U^n_k$ is a Jordan domain.  Set $\wt{X}^0_i:=f^n(\wt{X}^n_k)$, then it is an $0$-tile and  $f^n:\wt{X}^n_k\to \wt{X}^0_i$ is a conformal homeomorphism.  By the definition of $\mc{K}_n$, we have the conformal homeomorphism $f^n:U^n_k\to U^0_i$.

The proof goes by contradiction. Without loss of generality, we assume that $\{U^{n}_{k_n}=\tu{int}(\wt{X}^{n}_{k_n}\setminus \mc{K}_n)\}_{n\geq1}$ is a sequence of Jordan domains such that
\begin{equation}\label{converging}
\tu{diam}_\sigma(U^n_{k_n})\to C>0 \tu{ as } n\to\infty,
\end{equation}
and $f^{n}(U^n_{k_n})=U^0_0$ for all $n\geq1$. Consider the sequence of conformal maps $\{g_n:=f^{-n}|_{U^0_0}\}$. It is a normal family, because $\mc{K}_0\subseteq \mc{K}_{n}$ and hence their images avoid the set $\mc{K}_0$. We then obtain a subsequence of univalent maps, still written $\{g_n\}$, locally uniformly converging to a holomorphic map $g$ on $U^0_0$. By \eqref{converging} the map $g$ is not a constant. Thus $g:U^0_0\to g(U^0_0)$ is univalent.

We claim that the domain $g(U^0_0)$ is contained in the Fatou set. Otherwise, for any domain $V$ with $\ol{V}\subset g(U^0_0)$ and $V\cap \mathcal{J}_f\neq \es$, the iteration $f^n(V)$ will eventually cover the sphere except two points. This contradicts the fact that $f^{n}(V)\subset U^0_0$ for sufficiently large $n$.

Let $W$ be a domain with ${\rm cl}(W)\subset g(U^0_0)$. Note that each periodic Fatou component are supper attracting (since $f$ is postcritically finite), then the claim above implies that $f^{n}(W)$ converges to an attracting periodic orbit. On the other hand, as $n$ is large enough, the map $f^{n}|_W=g^{-1}_n|_W$ uniformly converges to the univalent map $g^{-1}|_W$. It is a contradiction.
\end{proof}

\begin{proposition}\label{expanding}
The Thurston map $F$ defined in \eqref{commute-1} is  expanding.
\end{proposition}
\begin{proof}
To show the expansion of $F$, we only have to prove that the maximum of  $\tu{diam}_\sigma(X^n)$ with $X^n\in\mf{X}^n(\CCC_F)$ converges to $0$ as $n\to 0$.
Note that $\pi(\wt{X}^n\setminus\mathcal{K}_n)=\pi(\wt{X}^n_k)$ for all $\wt{X}^n\in\mf{X}^n(\CCC_f)$. Then the proposition follows immediately from Lemmas \ref{corresponding} and  \ref{max_diameter}.
\end{proof}

\subsection{Proof of the main theorem}\label{main-proof}
\begin{proof}[Proof of Theorem \ref{main_thm}]
The proof follows the outline given in  Section\ \ref{outline}.

We first prove the sufficiency.  Let $f$ be a postcritically finite rational map with Sierpi\'{n}ski carpet Julia set. From  Section~\ref{expanding-quotient} we obtain a semi-conjugacy $\pi$ from the rational map $f$ to an expanding Thurston map $F$, and two Jordan curves $\CCC_f$ and $\CCC_F$.

We label post($f$) in $\CCC_f$ by $\wt{x}_1,\cdots,\wt{x}_m, \wt{x}_{m+1}=\wt{x}_1$ successively in the cyclic order, and denote by $\CCC_f(\wt{x}_i,\wt{x}_{i+1})$ the arc on $\CCC_f$ with endpoints $\wt{x}_i$ and $\wt{x}_{i+1}$. Set $x_i=\pi(\wt{x}_i)$ for all $i\in\{1,\ldots,m+1\}$ and similarly define $\CCC_F(x_i,x_{i+1})$.  It is clear that $\CCC_F(x_i,x_{i+1})=\pi(\CCC_f(\wt{x}_i,\wt{x}_{i+1}))$. Moreover, by Lemma \ref{corresponding} there is an one to one correspondence between $\mf{X}^n(\CCC_f)$ and $\mf{X}^n(\CCC_F)$, characterized by the map $\mf{X}^n(\CCC_f)\ni \wt{X}^n_k\mapsto X^n_k:=\pi(\wt{X}^n_k)\in \mf{X}^n(\CCC_F)$ for each $n\geq0$.

Let $\psi:\CCC_f\to \CCC_F$ be an orientation preserving homeomorphism such that $\psi(\wt{x}_i)=x_i$ and $\psi(\CCC_f(\wt{x}_i,\wt{x}_{i+1}))=\CCC_F(x_i,x_{i+1})$ for all $i\in\{1,\ldots,m\}$. There exists then a homotopy $h^0:\CCC_f\times I\to \CCC_F$ rel. ${\rm post}(f)$ from $\psi$ to $\pi|_{\CCC_f}$ by Lemma \ref{homotopy} (1).  We extend $\psi$ to an orientation preserving homeomorphism of $\widehat{\mathbb{C}}$, also denoted by $\psi$, with $\psi(\wt{X}^0_k)=X^0_k,k=0,1$. It follows from Lemma \ref{homotopy} (2) that the homotopy $h^0$ can be extended to a homotopy $H^0:S^2\times I\to S^2$ rel. ${\rm post}(f)$ from $\psi$ to $\pi$. By the property of $\pi$ that $\pi^{-1}(x)$ is either a point in $\mathcal{J}_f$ or the closure of a Fatou component, and the specific construction of homotopies in Lemma \ref{homotopy}, we have that $(H^{0}_t)^{-1}(x_i)=x_i,\forall t\in[0,1)$ and $(H^{0}_1)^{-1}(x_i)=[\wt{x}_i]$  for each $x_i\in {\rm post}(F)$.

We know that the Riemann sphere $\widehat{\mathbb{C}}$ and the $2$-sphere $S^2$ admit a partition by the $1$-tiles relative to $(f,\CCC_f)$ and $(F,\CCC_F)$ respectively, and the numbers of $\mf{X}^1(\CCC_f)$ and $\mf{X}^1(\CCC_F)$ are both $2d$.
For each $j\in\{1,\ldots,2d\}$, we define a map
$$\phi_j:=(F|_{X^1_j})^{-1}\circ \psi\circ f|_{\wt{X}^1_j}:\wt{X}^1_j\to X^1_j.$$
It is a composition of $3$ onto homemorphisms, and hence a homeomorphism. It also satisfies that $\phi_j(\partial \wt{X}^1_j)=\partial X^1_j$ and $\phi_j|_{f^{-1}(\textup{post}(f))}=\pi|_{f^{-1}(\textup{post}(f))}$. Using Lemma \ref{homotopy} again, we get a homotopy $H_j^1:\wt{X}^1_j\times I\to X^1_j$ rel. $f^{-1}(\textup{post}(f))$ from $\pi|_{\wt{X}^1_j}$ to $\phi_j$.

 We define the map $\phi:\widehat{\mathbb{C}}\to S^2$ by $\phi(z):=\phi_j(z)$ if $z\in \wt{X}^1_j$, and the map $H^1:\widehat{\mathbb{C}}\times I\to S^2$ by $H^1(z,t):=H^1_j(z,t)$ if $z\in \wt{X}^1_j, t\in[0,1]$. It is clear that $\phi$ is a homeomorphism and $H^1$ is a homotopy rel. $f^{-1}({\rm post}(f))$ from $\pi$ to $\phi$.  With the same reason, the homotopy $H^1$ satisfies the similar property as $H^0$, that is, $(H^{1}_t)^{-1}(p)=\wt{p},\forall t\in[0,1)$ and $(H^{1}_1)^{-1}(p)=[\wt{p}]$ for all $p\in F^{-1}({\rm post}(F))$, where $ \wt{p}\in f^{-1}({\rm post}(f))$ satisfy that $\pi(\wt{p})=p$.

Concatenating the homotopies $H^0, H^1$ together, we get a homotopy $H:\widehat{\mathbb{C}}\times I\to S^2$ rel. ${\rm post}(f)$ from $\psi$ to $\phi$ defined by
\[H(z,t)=\left\{
    \begin{array}{ll}
      H^0(z,2t), & \hbox{if $z\in \widehat{\mathbb{C}}, t\in[0,1/2]$;} \\[5pt]
      H^1(z,2t-1), & \hbox{if $z\in \widehat{\mathbb{C}}, t\in[1/2,1]$.}
    \end{array}
  \right.
\]
The homotopy $H$ satisfies that $H^{-1}_t(x_i)=x_i, \forall t\in[0,1]\setminus\{1/2\}$ and $H^{-1}_{1/2}(x_i)=[\wt{x}_i]$, for each $x_i\in {\rm post}(F)$.
By Lemma \ref{good-homo}, the homeomorphisms $\psi$ and $\phi$ are isotopic rel. ${\rm post}(f)$.  

We now turn to the necessity. Let $f$ be a postcritically finite rational map. By Whyburn's characterization (see  Section~\ref{rational map}) and Lemma \ref{small}, in order to show that $\mathcal{J}_f$ is a Sierpi\'{n}ski carpet, we just need to prove that
the closures of any two distinct Fatou components are disjoint and each Fatou component is a Jordan domain.

Let $f$ be Thurston equivalent to an expanding Thurston map $F$ via $h_0, h_1$. Using isotopy lifting theorem (see \cite[Proposition 11.1]{BM}) repeatedly, we obtain a sequence of homeomorphisms $\{h_n\}_{n\geq 0}$ such that $h_n\circ f=F\circ h_{n+1}$ and $h_n$ is isotopic to $h_{n+1}$ rel. $f^{-n}(\textup{post}(f))$, i.e., the following diagram commutes.
\begin{equation}\label{commute-3}
\begin{CD}
{\widehat{\mathbb{C}}} @> {\cdots} >> {\widehat{\mathbb{C}}}@>{f}>>{\widehat{\mathbb{C}}}@>{\widehat{\mathbb{C}}}>>{\widehat{\mathbb{C}}}@>{f}>>{\widehat{\mathbb{C}}}@>{f}>>{\widehat{\mathbb{C}}} \\
@ VV{\cdots}V @ VV {h_{n+1}} V@VV {h_n}V@VV {h_2}V@VV {h_1}V@VV {h_0}V \\
{S^2} @> {\cdots} >> {S^2}@>{F}>>{S^2}@>{\cdots}>>{S^2}@>{F}>>{S^2}@>{F}>>{S^2} \\
\end{CD}
\end{equation}
Since $F$ is expanding, by \cite[Lemma 11.3]{BM}, the sequence of homeomorphisms $\{h_n\}_{n\geq 0}$ uniformly converges to a continuous map $h$ with respect to a metric $\omega$ on $S^2$. We then get a semi-conjugacy $h$ from  $f$ to $F$, i.e., $F\circ h=h\circ f$ on $\widehat{\mathbb{C}}$. Besides, the restriction
\begin{equation}\label{bij_pre_post}
h:\cup_{n\geq 0} f^{-n}(\textup{post}(f))\to \cup_{n\geq 0}F^{-n}(\textup{post}(f))
\end{equation}
is bijective.
Because for all $k\geq n\geq 0$, $h_k=h_n:f^{-n}(\textup{post}(f))\to F^{-n}(\textup{post}(F))$ is bijective.

We claim that $h(\overline{U})$ is a singleton for each Fatou component $U$. By Sullivan's non-wandering Fatou component Theorem, we can assume that $U$ is fixed by $f$. Let $r'$ be a periodic internal ray of $U$ with period $q$, and denote the center of $U$ by $c$.
For each $n\geq 0$, we set $r_{n}:=h_{qn}(r')$, which is a lift of $r_{0}$ by $F^{qn}$ following the commutative diagram (\ref{commute-3}). By the expansion of $F$, it is proved in \cite[Lemma~8.7]{BM} that
$\textup{diam}_{\omega}(r_{n})\leq C{\Lambda^{-qn}}\to 0$
as $n\to \infty$, where $\Lambda\geq1$ and $C$ are constants. So $h(r')=h(c)$.
As the periodic internal rays are dense in $\overline{U}$, we get that $h(\overline{U})=h(c)$.
By this claim and  \eqref{bij_pre_post}, the closures of distinct Fatou components are disjoint.

We are left to show that each Fatou component is a Jordan domain. Without loss of generality, let $U$ be a fixed Fatou component. We assume that $U$ is not a Jordan domain and argue by contradiction. Note that $\partial U$ is locally connected, from the \emph{B\"{o}ttcher's theorem} there exist two internal rays of $U$ landing at a common point in $\partial U$.  The closure of their union is a Jordan curve bounding two domains $W_0$ and $W_1$ with $W_i\cap \mathcal{J}_f\neq \emptyset$, $i\in\{0,1\}$.

We claim that each of the domains $W_0,W_1$ contains a Fatou component. Otherwise, there is $i\in\{0,1\}$  such that $f^n(W_i)\subseteq (U\cup \mathcal{J}_f)$ for all $n\geq 0$. By the topological transitivity of the Julia set, the set $f^n(W_i)$ for sufficiently large $n$, hence $U\cup \mathcal{J}_f$, covers $\widehat{\mathbb{C}}$ except at most two points (see \cite[Theorem 4.10]{Mi1}). It means that $f$ has only one Fatou component $U$, impossible.

Let $U_0$ and $U_1$ be the Fatou components contained in $W_0$ and $W_1$ respectively. By the discussion above, we have that the images $h(\overline{U}),h(\overline{U_0})$ and $h(\overline{U_1})$ are pairwise different points. Consequently, the set $h^{-1}(h(\overline{U}))$ contains the Jordan curve $\partial W_0=\partial W_1\subset \overline{U}$, and is disjoint with $U_0,U_1$. It implies that $S^2\setminus h^{-1}(h(\overline{U}))$ is not connected. On the other hand,
note that $h$ is the limit of a sequence of homeomorphisms of $S^2$. By \cite[Lemma 3.1]{CPT}, such a map $h$ has a property that $S^2\setminus {h^{-1}(x)}$ is connected for any $x\in S^2$. It contradicts that $S^2\setminus h^{-1}(h(\overline{U}))$ is not connected.  The proof of the necessity is completed.
\end{proof}

\section{Appendix}\label{sec-appendix}
\begin{proof}[Proof of Lemma \ref{good-homo}]
Let $H:S\times I\to S$ rel.~$\partial S\cup P$ be a homotopy from $id_S$ to $h$. Set $K:=\cup_{p\in P}K_p$. See Lemma \ref{good-homo} for the definition of $K_p$.  We can choose a closed topological disk $D_p\subset S$ for each $p\in P$ such that

  $\bullet$ $D_p\cap P=\{p\}$;

  $\bullet$ $D_p\cap D_q=\emptyset$ if $p\not=q\in P$ and

  $\bullet$ $\gamma_p:=\partial D_p\subset S\setminus K$.

We claim that it is enough to prove the lemma in the case that $h|_{\gamma_p}=id_{\gamma_p}$ for all $p\in P$. To see this, note first that the homotopy $H$ induces a homotopy
$H|_{\gamma_p\times I}:\gamma_p\times I\to S\setminus P$ from $\gamma_p$ to $h(\gamma_p)$ for each $p\in P$. By \cite[Theorem~2.1]{Ep}, there exists an isotopy $$\phi:(S\setminus P)\times I\to S\setminus P \text{ rel}.~\partial(S\setminus P)$$
 such that $\phi_0=id_{S\setminus P}$ and $\phi_1|_{\gamma_p}=h|_{\gamma_p}$ for all $p\in P$. It can be also viewed as an isotopy  $\phi: S\times I\to S$ rel.~$\partial S\cup P$. We then get a homotopy $\Phi:S\times I\to S$ rel.~$\partial S\cup P$, defined by $\Phi(x,t):=\phi_t^{-1}\circ H_t(x)$, from $id_S$ to $\phi_1^{-1}\circ h$.  The map $\phi_1^{-1}\circ h$  is identity when restricted on $\cup_{p\in P}\gamma_p$.  If we proved that $id_S\overset{{\rm iso.}}{\thicksim}\phi_1^{-1}\circ h$ rel.~$\partial S\cup P$, it follows that
$id_S\overset{{\rm iso.}}{\thicksim}\phi_1\overset{{\rm iso.}}{\thicksim} h$ rel.~$\partial S\cup P$.

We now assume that $h|_{\gamma_p}=id_{\gamma_p}$ for all $p\in P$. We claim that if each $\gamma_p$ is fixed by the homotopy $H$, i.e., $H(x,t)=x$ for all $t\in[0,1],x\in\gamma_p$, then the conclusion holds. Let the homotopy $H$ satisfy the property of this claim. Applying $(2)$ in Lemma \ref{homotopy} to $h|_{D_p}$ and $id_{D_p}$ with $A=\gamma_p$, we have that $h|_{D_p}$ is isotopic to $id_{D_p}$ rel.~$\gamma_p$ for each $p\in P$. Set $M:=S\setminus(\cup_{p\in P}{\rm int}(D_p))$. For each $p\in P$, we construct a radical projection
\[\pi_p:D_p\setminus\{p\}\to \gamma_p~~~z\mapsto\alpha_p^{-1}(\alpha_p(z)/|\alpha_p(z)|),\]
where $\alpha_p:D_p\to {\rm cl}(\D)$ is a homeomorphism with $\alpha_p(p)=0$. Note that for any $x\in M$, the curve $H(x,t),t\in I$ avoids $P$. Then we define a map  $\widetilde{H}:M\times I\to S$ by
\[\widetilde{H}(x,t):=\left\{
    \begin{array}{ll}
      \pi_p\circ H(x,t), & \hbox{if $H(x,t)\in D_p$ for some $p\in P$;} \\[5pt]
      H(x,t), & \hbox{otherwise.}
    \end{array}
  \right.
\]
This map is continuous and satisfies that
\begin{itemize}
  \item $\widetilde{H}_0=id_M$ and $\widetilde{H}_1=h|_M$;
  \item $\widetilde{H}(x,t)=x$ for $x\in \cup_{p\in P}\gamma_p,t\in[0,1]$; and
\item $\widetilde{H}_t(M)\subset M$ for $t\in[0,1]$.
\end{itemize}
 Thus we get a homotopy $\widetilde{H}|_{M\times I}:M\times I\to M$ rel.~$\partial M$. By Theorem \ref{homo-iso}, the maps $id_M$ and $h|_{M}$ are isotopic rel.~$\partial S$. From the argument above, we see that $id_S$ is isotopic to $h$ rel. $\partial S\cup P$.

Consequently, we remain to find a homotopy $\widehat{H}:S\times I\to S$ rel.~$\partial S\cup P\cup(\cup_{p\in P}\gamma_p)$ from $id_S$ to $h$ under the assumption that $h|_{\gamma_p}=id_{\gamma_p}$ for all $p\in P$.

Cutting the disks $D_p,p\in P$ from $S$, we get the surfaces $D_p,p\in P$ and $M$ following the notations above. Note that each $\gamma_p$ belongs to both $D_p$ and $M$. For distinguishing, we denote the $\gamma_p$ in $D_p$ by $\gamma_p^-$ and that in $M$ by $\gamma_p^+$. And a point $\xi\in\gamma_p$ is represented by $\xi^{\pm}$ in $\gamma_p^{\pm}$. We then paste each $D_p$ with $M$ by the annulus $A_p:=\gamma_p\times[-1,1]$. Precisely, let $\approx$ be an equivalence relation on the disjoint union $(\sqcup_{p\in P} D_p) \sqcup (\sqcup_{p\in P}A_p) \sqcup M$ such that $x\approx y$ if and only if $x=y$, or $x=\xi^{\pm}\in\gamma_p^{\pm}$ and $y=\xi\times\{\pm1\}\in A_p$ for some $\xi\in\gamma_p$ and $p\in P$.
The quotient space
$$S_b:=(\bigsqcup_{p\in P} D_p \sqcup \bigsqcup_{p\in P} A_p \sqcup M)/_\approx$$
is clearly homeomorphic to $S$.

For a point $x\in S_b$, it has a natural name $x$ as in the original surface $S$ if $x\in(\sqcup_{p\in P} D_p) \sqcup M$; and it is parameterized by
$\{(\xi,s):\xi\in \gamma_p, s\in [-1,1]\}$ if $x\in A_p$ for some $p\in P$.
With these notations, we define a map $H_b:S_b\times I\to S$ by
\[H_b(x,t):=\left\{
              \begin{array}{ll}
                H(x,t), & \hbox{if $x\in (\sqcup_{p\in P} D_p) \sqcup M$;} \\[5pt]
                H(\xi, t|s|), & \hbox{if $x=(\xi,s)\in A_p$ for some $p\in P$.}
              \end{array}
            \right.
\]
Then the map $H_b:S_b\times I\to S$ is a homotopy rel.~$\partial S\cup P\cup (\cup_{p\in P}\gamma_p\times\{0\})$.

As $S_b$ is homeomorphic to $S$, we identify $S_b$ with $S$ and identify a point $\xi\times 0\in\gamma_p\times0\subset A_p$ with $\xi\in\gamma_p\subset S$ for $p\in P$. In such a view, each $A_p$ is an annulus neighborhood of $\gamma_p$, and $H_b:S\times I\to S$ is a homotopy relative to $\partial S\cup P\cup (\cup_{p\in P}\gamma_p)$.

By the definition of $H_b$, the map $H_b|_{S\times0}$ (\emph{resp.} $H_b|_{S\times1}$) is homotopic to $id_S$ (\emph{resp.} $h$) rel. $\partial S\cup P\cup (\cup_{p\in P}\gamma_p)$. It follows that
\[id_S\overset{homotopy}{\thicksim}H_b|_{S\times0}\overset{~H_b~}{\thicksim} H_b|_{S\times1}\overset{homotopy}{\thicksim} h~{\rm rel}.~\partial S\cup P\cup (\cup_{p\in P}\gamma_p).\]
Then the proof is completed.
\end{proof}


\vspace{0.2cm}

\noindent Yan Gao, \\
Mathemaitcal School  of Sichuan University, Chengdu 610064,
P. R. China. \\
Email: gyan@scu.edu.cn
\vspace{0.2cm}

\noindent Jinsong Zeng, \\
Academy of Mathematics and Systems Science, \\
Chinese
Academy of Sciences, Beijing 100190, P. R. China. \\
Email: jinsong.zeng@amss.ac.com
\vspace{0.2cm}

\noindent Suo Zhao,  \\
Mathemaitcal School  of Sichuan University, Chengdu 610064,
P. R. China. \\
 Email: saaki@163.com

\end{document}